\documentclass[12pt]{amsart}
\usepackage[utf8]{inputenc}
\usepackage{amsmath,amsthm,amssymb,amsfonts,enumerate,color}
\usepackage{mathrsfs}
 \usepackage{amstext,amsxtra}
 \usepackage{txfonts} %also pxfonts
 \usepackage[colorlinks, citecolor=blue,pagebackref,hypertexnames=false]{hyperref}
 \usepackage{pgf,tikz}
  \usepackage{bbm}
    \usepackage{soul}

\allowdisplaybreaks

\def\zz{\mathbb{Z}}
\def\d{\mathrm{d}}
\def\n{\mathbf{n}}
\def\m{\mathbf{m}}

\begin{document}

\newcommand{\norm}[1]{\left\Vert#1\right\Vert}
\newcommand{\abs}[1]{\left\vert#1\right\vert}
\newcommand{\set}[1]{\left\{#1\right\}}
\newcommand{\Real}{\mathbb{R}}
\newcommand{\RR}{\mathbb{R}^n}
\newcommand{\supp}{\operatorname{supp}}
\newcommand{\card}{\operatorname{card}}
\renewcommand{\L}{\mathcal{L}}
\renewcommand{\P}{\mathcal{P}}
\newcommand{\T}{\mathcal{T}}
\newcommand{\A}{\mathbb{A}}
\newcommand{\K}{\mathcal{K}}
\renewcommand{\S}{\mathcal{S}}
\newcommand{\blue}[1]{\textcolor{blue}{#1}}
\newcommand{\red}[1]{\textcolor{red}{#1}}
\newcommand{\Id}{\operatorname{I}}

\newtheorem{thm}{Theorem}[section]
\newtheorem{prop}[thm]{Proposition}
\newtheorem{cor}[thm]{Corollary}
\newtheorem{lem}[thm]{Lemma}
\newtheorem{lemma}[thm]{Lemma}
\newtheorem{exams}[thm]{Examples}
\theoremstyle{definition}
\newtheorem{defn}[thm]{Definition}
\newtheorem{rem}[thm]{Remark}
\newtheorem{cj}[thm]{Conjecture}
\numberwithin{equation}{section}

\title[fractional discrete Laplacian \& Riesz potential]
{The derivative of the fractional discrete Laplacian \\ is an exotic Riesz potential}

\author[Bo Li, Qingze Lin, Huoxiong Wu]{Bo Li, Qingze Lin, Huoxiong Wu$^\ast$}
\thanks{$\ast$ Corresponding author}

\address{B. Li, College of Data Science, Jiaxing University,    Jiaxing 314001, China}
 \email{bli@zjxu.edu.cn}
%
%%\address{J. Li, School of Mathematics and Information Science, Henan Polytechnic University, Jiaozuo 454003, China}
%% \email{jinxiali@hpu.edu.cn}

  \address{Q. Lin, Department of Mathematics, Shantou University,
         Shantou 515063, China}
 \email{gdlqz@e.gzhu.edu.cn}

  \address{H. Wu, School of Mathematical Sciences, Xiamen University, Xiamen, 361005, China}
 \email{huoxwu@xmu.edu.cn}

%
%
%\address{M. Song, School of Mathematics and Statistics, Northwestern Polytechnical University, Xi'an,710129, China}
%\email{mlsong@nwpu.edu.cn}

%
%\address{School of Mathematics\\ Sun Yat-sen University \\
%         Guangzhou 510275, P.~R. China}
% \email{linqz@mail2.sysu.edu.cn}
 \subjclass[2010]{39A12, % Discrete version of topics in analysis
35R11, %Fractional partial differential equations
39A12.%Discrete version of topics in analysis,
 }
\keywords{Discrete harmonic analysis, Fractional Laplacian, Riesz potential}

\begin{abstract}
Let $\Delta_{N}$ be the multidimensional discrete Laplacian on $\mathbb{Z}^N$ ($N\ge1$).
In this note, we prove that,
when $N=1$,  the right hand derivative of $(-\Delta_1)^s$ at $0$ is an exotic discrete Riesz potential
(namely, the endpoint case: the order is 0) in Stein-Wainger sense (J. Anal. Math. 2000),
and when $N\ge 2$, the corresponding derivative is also an exotic discrete Riesz potential with an additional corrector.
A similar conclusion for the left hand derivative case is also considered.
All results obtained in this note extend the logarithmic Laplacian of Chen-Weth (Comm. PDEs. 2019) to the discrete setting.
%
%a stratified Lie group $\mathbb{G}$\,.\, In this note, by utilizing the functional calculus, we bypass the lack of the harmonicity of $-\Delta_{\mathbb{G}}$ and then extend the work of Fabes et al. in {\it Duke Math. J.} concerning the Carleson measure characterizations of $BMO$ spaces, which are not involving with time derivatives of heat kernels, to the case of stratified Lie groups.\, In particular, this partially answers a conjecture posed by Lin in {\it Proc. AMS} in an affirmative way.
 \end{abstract}

\maketitle

\section{\bf Introduction}\label{s1}

In recent years, a great deal of mathematical effort in boundary value problems involving (non-)linear integro-differential operators
has been devoted to the study of the fractional power of the Laplace operator.
For any $0<s<1$, the fractional Laplacian $(-\Delta)^s$ on $\mathbb{R}^d$ is defined via Fourier transform as
$$
\widehat{(-\Delta)^sf}(\xi)=|\xi|^{2s}\widehat{f}(\xi),\quad f\in C_0^\infty(\mathbb{R}^d),
$$
and it can be expressed by the pointwise formula
$$
(-\Delta)^sf(x)=c_{d,s}\mathrm{P.V.}\int_{\mathbb{R}^d}\frac{f(x)-f(y)}{|x-y|^{d+2s}}\d y.
$$
It is well known that the fractional Laplacian $(-\Delta)^s$ admits  the following limiting property  
\begin{align}\label{q1}
(-\Delta)^sf(x)\to f(x)
\end{align}
when $s$ converges to zero.
A challenging problem is
\begin{center}
\textit{whether the expansion \eqref{q1} could be extended to the first order or high order.}
\end{center}
A partial affirmative answer for this problem is given in \cite{CW2019} that, for any $C^2$-function $f$ on $\mathbb{R}^d$ with compact support,
$$
(-\Delta)^sf(x)=f(x)+sL_\Delta f(x)+o(s),\quad s\to0^+,
$$
where the operator $L_\Delta$ is given as a logarithmic Laplacian formally,
and is regarded as the right derivative of $(-\Delta)^s$ at $s=0$ %\footnote{Generally speaking, the derivative of $T_tf$}
$$
L_\Delta =\left.\frac{\d}{\d s}(-\Delta)^s\right|_{s=0^+}.
$$
%see \cite{CL2022-2,CL2022-1,KL1988,LXY2020} for the spatial derivative of the maximal function.
Some qualitative properties for $L_\Delta$ are obtained as follows.
\begin{flushleft}
{\bf{Theorem A.}} (\cite[Theorem 1.1]{CW2019})
{\it{
Let $f\in C^\alpha(\mathbb{R}^d)$ for some $\alpha>0$ with compact support. Then we have
$$
L_\Delta f(x)=\left.\frac{\d}{\d s}(-\Delta)^s\right|_{s=0^+} f(x)=c_d\int_{\mathbb{R}^d}\frac{f(x)\mathbbm{1}_{B_1(x)}(y)-f(y)}{|x-y|^{d}}\d y+\rho_df(x).
$$
Here $c_d=\pi^{-d/2}\Gamma(d/2)$ and $\rho_d=\log4+\psi(d/2)-\gamma$,
where $\psi=\Gamma^\prime/\Gamma$ is the Digamma function and $\gamma$ is the Euler-Mascheroni  constant.
Moreover,
\begin{enumerate}
  \item for $1<p\le \infty$, we have $L_\Delta f\in L^p(\mathbb{R}^d)$ and $[(-\Delta)^sf-f]/s\to L_\Delta f$ in $ L^p(\mathbb{R}^d)$ as $s\to0^+$;
  \item $\widehat{L_\Delta f}(\xi)=(2\log|\xi|)\widehat{f}(\xi)$ for a.e. $\xi\in \mathbb{R}^d$.
\end{enumerate}
}}
\end{flushleft}

The study of the fractional concept has been widely applied to different settings
such as the obstacle problem \cite{BFS2018,CSS2008,HT2024}, conformal geometry \cite{CG2011}, trace/extension problem \cite{CS2007,LLLMS,ST2010},
Fourier transform \cite{CFGW2021}, elliptic/parabolic equation \cite{CW2021,CM2023,DFX2018},
Nirenberg problem \cite{JLX2014}, Yamabe problem \cite{ACDFGW2019,JX2014}, differential operation \cite{CDH2016,LHZ2023,LSX2024}.

This note investigates a discrete version of Theorem A on $\zz^N$ ($N\ge1$).
When $N=1$, we prove that the right derivative of the fractional discrete Laplacian at zero is the discrete Riesz potential of order zero (see Section \ref{s2} below),
and when $N\ge2$, it is an exotic discrete Riesz potential with a corrector (see Section \ref{s3} below).
Here we remark that, when $N=1$, the fractional discrete Laplacian is determined by the Gamma function, and it has an explicit representation.
However, when $N\ge2$, the kernel of the fractional discrete Laplacian is not explicit,
and hence some methods and techniques in \cite{CW2019} are no longer applicable due to the complex structure of the underlying space $\zz^N$.
Finally, all results of this note as regards $(-\Delta_N)^s$ ($0<s<1$) can be extended to the reversed case $(-\Delta_N)^{s}$ ($-N/2<s<0$),
namely, in the discrete case $\zz^N$, the left derivative of the Riesz potential at zero is still the discrete Riesz potential of order zero;
see Section \ref{s4} below for more details.
%However, to keep the length of the note, we will not address it.

All constants with subscripts, such as $\rho_N$ (which depends on $N$ only), do not change in different occurrences.

\section{\bf The derivative of the fractional Laplacian on $\zz$}\label{s2}
\subsection{Preliminaries}
In this subsection, we consider a mesh of fixed size $h>0$ on $\mathbb{R}$ given by $\zz_h=\{hn: n\in\zz\}$.
For a function $f :\ \zz_h\to \mathbb{R} $, we use the notation $f_h(n) = f(hn)$ to denote the value of $f$ at the mesh point $hn\in \zz_h$.
The discrete Laplacian $\Delta_h$ on $\zz_h$ is then given by
$$
-\Delta_hf_h(n)=-\frac{1}{h^2}(f_h(n+1)+f_h(n-1)-2f_h(n)),
$$
and its fractional power $(-\Delta_h)^s$ $(0<s<1)$ can be defined by the heat semigroup method (\cite{CRSTV2018,ST2010})
\begin{align*}
(-\Delta_h)^sf_h=\frac{1}{\Gamma(-s)}\int_0^\infty(e^{t\Delta_h}f_h-f_h)\frac{\d t}{t^{s+1}},
\end{align*}
where $e^{t\Delta_h}$ is the heat semigroup associated with $-\Delta_h$, and $u_h(n,t)=e^{t\Delta_h}f_h(n)$ is the solution to the semidiscrete heat equation
$$
\begin{cases}
\partial_t u_h=\Delta_hu_h, \ &\textrm{in }\zz_h\times(0,\infty),\\
u_h(\cdot,0)=f_h, &\textrm{on }\zz_h.
\end{cases}
$$

For any $0\le s\le1$, the weighted Lebesgue space, denoted by $\ell_{ s}(\zz)$, is defined as the class of all functions $f:\zz_h\to\mathbb{R}$ such that their norms
$$
\|f_h\|_{\ell_{ s}(\zz)}=\sum_{n\in\zz}\frac{|f_h(n)|}{(1+|n|)^{1+ 2s}}<\infty.
$$
A result about the fractional Laplacian in \cite{CRSTV2018} tells us that, if $f_h\in \ell_{s}(\zz)$ with $0<s<1$, then
$$
(-\Delta_h)^sf_h(n)=\sum_{m\in\zz,m\neq n}\mathcal{K}^h_s(n-m)(f_h(n)-f_h(m)),
$$
where the discrete kernel $\mathcal{K}^h_s(m)$ is given by
$$
%(-\Delta_h)^sf_h(n)=\sum_{m\in\zz,m\neq n}\mathcal{K}^h_s(n-m)(f_h(n)-f_h(m))
\mathcal{K}^h_s(m)
=\pi^{-1/2}s\left(\frac{2}{h}\right)^{2s}\frac{\Gamma(1/2+s)}{\Gamma(1-s)}\frac{\Gamma(|m|-s)}{\Gamma(|m|+1+s)} \mathbbm{1}_{\zz\setminus\{0\}}(m)
=C_h(s)\frac{\Gamma(|m|-s)}{\Gamma(|m|+1+s)}\mathbbm{1}_{\zz\setminus\{0\}}(m).
%&c_h(s)=\frac{C_h(s)}{s}=\pi^{-1/2}\left(\frac{2}{h}\right)^{2s}\frac{\Gamma(1/2+s)}{\Gamma(1-s)},\quad c_h(0)=1\\
%&c_h^\prime(0)=\log4-\log h^2+\psi(1/2)+\psi(1)=-2\gamma-\log h^2,
$$
Denote by $c_h(s)$ the quotient $C_h(s)/s$  for simplicity.
Obviously, one has $c_h(0)=1$ and
\begin{align}\label{q22}
c_h^\prime(0)=\log4-\log h^2+\psi(1/2)+\psi(1)=-2\gamma-\log h^2.
\end{align}
If $h=1$, then we throw away the superscript of $\mathcal{K}^h_s(m)$ and subscript of $\zz_h$.

Moreover, the corresponding logarithmic Laplacian $\log(-\Delta_h)$ is defined by the right hand derivative of $(-\Delta_h)^s$ at $0$, namely,
$$
\log(-\Delta_h)=\left.\frac{\d }{\d s}(-\Delta_h)^s\right|_{s=0^+}.
$$

\subsection{\bf Estimates for Gamma functions}

This subsection collects some properties for Gamma functions.
Here we provide the proofs for the sake of completeness though their arguments are standard.

\begin{prop}\label{prop1}
Let $s>0$ be sufficiently small.
For any integer $k\ge 1$, the following formulas hold.
\begin{enumerate}
  \item $\displaystyle\frac{\Gamma(k-s)}{\Gamma(k+1+s)}\le \frac{1}{(k-s)^{1+2s}}$.
  \item $\displaystyle\sum_{|m|\ge k} \frac{\Gamma(|m|-s)}{\Gamma(|m|+1+s)}=\frac{\Gamma(k-s)}{s\Gamma(k+s)}$.
  \item $\displaystyle \frac{\Gamma(k-s)}{\Gamma(k+s)}=1-2s\psi(k)+o(s)$.
  %\item $\displaystyle\frac{\Gamma(k+s)}{\Gamma(k+1-s)}\le \frac{k+1-s}{(k+s)^{2-2s}}$.
%  \item $\displaystyle\sum_{|m|<k} \frac{\Gamma(|m|+s)}{\Gamma(|m|+1-s)}=\frac{\Gamma(k+s)}{s\Gamma(k-s)}$.
%  \item $\displaystyle \frac{\Gamma(k+s)}{\Gamma(k-s)}=1+2s\psi(k)+o(s)$.
\end{enumerate}
\end{prop}

\begin{proof}
(1) It follows from the identity for the quotient of the Gamma function (see \cite[(5.1)]{CRSTV2018}) that
$$
\frac{\Gamma(k-s)}{\Gamma(k+1+s)}
=\frac{1}{\Gamma(1+2s)}\int_0^\infty e^{-(k-s)t}(1-e^{-t})^{2s}\d t\le \frac{1}{\Gamma(1+2s)}\int_0^\infty e^{-(k-s)t}t^{2s}\d t= \frac{1}{(k-s)^{1+2s}},
$$
where we used the basic inequality $1-e^{-t}\le t$ for any $t\ge 0$.

(2)  One may invoke the relationship between the Gamma function $\Gamma(\cdot)$ and the Beta function $B(\cdot,\cdot)$ to obtain
\begin{align*}
\sum_{|m|\ge k}  \frac{\Gamma(|m|-s)}{\Gamma(|m|+1+s)}
%&=\frac{2}{\Gamma(1+2s)}\sum_{m=k}^\infty \frac{\Gamma(m-s)\Gamma(1+2s)}{\Gamma(m+1+s)}\\
&=\frac{2}{\Gamma(1+2s)}\sum_{m=k}^\infty B(m-s,1+2s)\\
&=\frac{1}{s\Gamma(2s)} \int_0^1 \sum_{m=k}^\infty t^{m-s-1}(1-t)^{2s}\d t
%&=\frac{1}{\Gamma(1+2s)} \int_0^\infty  t^{k-s-1}(1-t)^{2s-1}\d t\\
%&=\frac{1}{\Gamma(1+2s)} B(k-s,2s) \\
=\frac{\Gamma(k-s)}{s\Gamma(k+s)}.
\end{align*}

(3) The behaviour of the quotient is a consequence of L'Hospital's rule.
%
%
%(4) Note that
%\begin{align*}
%\frac{\Gamma(k+s)}{\Gamma(k+1-s)}=(k+1-s)\frac{\Gamma(k+s)}{\Gamma(k+2-s)}.
%\end{align*}
%The remainder of the argument is analogous to that in (1) and is left to the reader.
%
%(5) This summation can be deduced easily by using induction.
%
%(6) The proof is similar to that in (3). We leave the detail to the read.
This completes the proof.
\end{proof}

\subsection{\bf The main result and its proof}
In this subsection, we will prove the first main result of this note.
\begin{thm}\label{thm1}
Let $f$ be a real-value function on $\zz_h$ ($h>0$) with bounded support.
Then we have
\begin{align}\label{q2}
\log(-\Delta_h)f_h(n)=\left.\frac{\d }{\d s}(-\Delta_h)^s\right|_{s=0^+}f_h(n)=-\sum_{m\in\zz, m\neq n}\frac{f_h(m)}{|n-m|} -(\log h^2)f_h(n).
\end{align}
Moreover, we have $\log(-\Delta_h)f_h\in \ell^\infty(\mathbb{Z})$ and
$$\lim_{s\to0^+}\frac{(-\Delta_h)^sf_h-f_h}{s}=\log(-\Delta_h)f_h$$
in $L^\infty$-sense.\footnote{In fact, the $L^\infty$-norm can be replaced by the $L^p$-norm with $1<p\le \infty$, and the proof is left to the readers.}
\end{thm}

\begin{rem}
In Theorem \ref{thm1}, we adopt the notation ``log'' to denote the derivative of the fractional discrete Laplacian at zero, namely,
$$
\log(-\Delta_h)=\left.\frac{\d }{\d s}(-\Delta_h)^s\right|_{s=0^+}.
$$
There are two reasons for this notation.
The first one is that the logarithmic function appears naturally in the differential operation (at least formally)
$$
\left.\frac{\d }{\d s}(-\Delta_h)^s\right|_{s=0^+}=\left.(-\Delta_h)^s\log(-\Delta_h) \right|_{s=0^+}=\log(-\Delta_h).
$$
The other is that, in the Euclidean case, the (weakly) singular integral operator $L_\Delta$ has a Fourier symbol of the form $\xi \mapsto 2\log|\xi|$; see Theorem A.
Moreover, in the discrete case, the operator $\log(-\Delta_h)$ is nonsingular (since $K^h_s(0)=0$)
and admits a Fourier symbol of the form $\xi \mapsto 2\log |(2\sin\pi\xi)/h|$ at least formally.
Indeed, for any $f_h\in\ell^1(\zz)$, its Fourier transform is given by
$$
\widehat{f_h}(\xi)=\sum_{n\in\zz}f_h(n)e^{2\pi in\xi}, \quad -1/2\le \xi\le 1/2;
$$
see \cite[Section 3]{Ci2023} for more details.
A routine computation gives rise to
\begin{align*}
\widehat{(-\Delta_hf_h)}(\xi)
&=\sum_{n\in\zz}\frac{-1}{h^2}(f_h(n+1)+f_h(n-1)-2f_h(n))e^{2\pi in\xi}\\
&=\sum_{n\in\zz}f_h(n)e^{2\pi in\xi}\frac{-1}{h^2}(e^{2\pi i\xi}+e^{-2\pi i\xi}-2)=\left(\frac{2}{h}\sin\pi\xi\right)^2\widehat{f_h}(\xi).
\end{align*}
From this it is only one step to replace the exponent 2 by a general exponent $2s$,
and thus to define (at least formally) the fractional power of the discrete Laplacian by
$$
\widehat{(-\Delta_h)^sf_h}(\xi)=\left(\frac{2}{h}\sin\pi\xi\right)^{2s}\widehat{f_h}(\xi).
$$
The sine function of the Fourier symbol above does not appear in the Euclidean case.
We believe that this phenomenon is caused by the Fourier transform of the heat kernel
$$
\widehat{G_{t,1}}(\xi)=e^{-4t\sin^2\pi \xi}.
$$
However, we still employ the notation $\log(-\Delta_h)$ for simplicity, rather than $\log\sin(-\Delta_h)$.
\end{rem}

\begin{rem}
It is convenient to introduce the discrete kernel function
\begin{align*}
K:\ &\zz\to \mathbb{R}\\
&m\mapsto|m|^{-1}\mathbbm{1}_{\zz\setminus\{0\}}(m).
\end{align*}
Then the summation representation \eqref{q2} can be rewritten as
\begin{align*}
\log(-\Delta_h)f_h(n)
&=-\sum_{m\in\zz, m\neq n}\frac{f_h(m)}{|n-m|} -(\log h^2)f_h(n)\\
&=-K\ast f_h(n)-(\log h^2)f_h(n)=-\mathcal{I}_0f_h(n)-(\log h^2)f_h(n) ,
\end{align*}
where $\mathcal{I}_0$ denotes  an exotic discrete Riesz potential (namely, the discrete Riesz potential of order zero)
with the kernel $K(n)=|n|^{-1}$; see \cite{SW2000} for example.
Moreover, we remark that $h=1$ is a special and important case in \eqref{q2}.
In this case, the derivative of the fractional discrete Laplacian reduces to the exotic discrete Riesz potential authentically.
\end{rem}

\begin{proof}[Proof of Theorem \ref{thm1}]
\textbf{Step 1}: decompose $(-\Delta_h)^sf_h$.
Pick an integer $\ell\ge 1$ sufficiently large such that $\supp f_h\subset (-\ell,\ell)$.
For any $s>0$ sufficiently small, one writes
\begin{align*}
(-\Delta_h)^sf_h(n)
%&=\sum_{m\in\zz,m\neq n}\mathcal{K}^h_s(n-m)(f_h(n)-f_h(m)) \\
&=\sum_{0<|m-n|<4\ell}\mathcal{K}^h_s(n-m)(f_h(n)-f_h(m))+\sum_{|m-n|\ge4\ell}\mathcal{K}^h_s(n-m)(f_h(n)-f_h(m)) \\
&=\left(\sum_{0<|m-n|<4\ell}\mathcal{K}^h_s(n-m)(f_h(n)-f_h(m))-\sum_{|m-n|\ge4\ell}\mathcal{K}^h_s(n-m)f_h(m)\right)+\sum_{|m-n|\ge4\ell}\mathcal{K}^h_s(n-m)f_h(n) \\
&=A^h_\ell(s,n)+D^h_\ell(s)f_h(n).
\end{align*}

\textbf{Step 2}: asymptotic behavior of $A^h_\ell(s,n)/s$ near the origin.
When $|n|<2\ell$, it follows from the support of $f_h$ that
$$
A^h_\ell(s,n)=\sum_{0<|m-n|<4\ell}\mathcal{K}^h_s(n-m)(f_h(n)-f_h(m)).
$$
Note that
$$
\frac{\mathcal{K}^h_s(n-m)}{s}=c_h(s)\frac{\Gamma(|n-m|-s)}{\Gamma(|n-m|+1+s)}\to \frac{1}{|n-m|},\quad s\to 0^+,
$$
and hence
$$
\frac{A^h_\ell(s,n)}{s}\to \sum_{0<|m-n|<4\ell}\frac{f_h(n)-f_h(m)}{|n-m|}=a^h_\ell(n),\quad s\to 0^+.
$$

\textbf{Step 3}: estimate $A^h_\ell(s,n)/s$ away from the origin.
When $|n|\ge 2\ell$, we note that $f_h(n)=0$ (which is guaranteed by the support of $f_h$) implies
$$
A^h_\ell(s,n)=-\sum_{m\in\zz,m\neq n}\mathcal{K}^h_s(n-m)f_h(m)
$$
and
$$
|n-m|\ge |n|-|m|>|n|/2\ge \ell
$$
whenever $|m|<\ell$.
From this, Proposition \ref{prop1} (1), it holds
\begin{align*}
\left|\frac{A^h_\ell(s,n)}{s}\right|
&\le \frac{C_h(s)}{s}\sum_{m\in\zz,m\neq n}\frac{\Gamma(|n-m|-s)}{\Gamma(|n-m|+1+s)}|f_h(m)|\\
&\le c_h(s)\sum_{m\in\zz,m\neq n}\frac{1}{(|n-m|-s)^{1+2s}}|f_h(m)|
%&\le \max_{0<s<1/2}c_h(s)\sum_{m\in\zz,m\neq n}|f_h(m)| \frac{1}{(N-s)^{1+2s}}\\
\le c_h(s)\|f_h\|_{\ell^1(\zz)}(\ell-s)^{-1}.
\end{align*}

\textbf{Step 4}: estimate $a^h_\ell(n)$ away from the origin.
When $|n|\ge 2\ell$, proceeding as the argument of Step 3, we have $f_h(n)=0$ (which is guaranteed by the support of $f_h$) and hence
$$
|a^h_\ell(n)|\le \sum_{0<|m-n|<4\ell}\frac{|f_h(m)|}{|n-m|}\le \|f_h\|_{\ell^1(\zz)}\ell^{-1}.
$$

\textbf{Step 5}: asymptotic behavior of  $(D^h_\ell(s)-1)/s$  near the origin. One may invoke Proposition \ref{prop1} (2) to deduce
$$
D^h_\ell(s)
%&=\sum_{|m-n|\ge4\ell}\mathcal{K}^h_s(n-m)\\
= C_h(s)\sum_{|m|\ge4\ell} \frac{\Gamma(|m|-s)}{\Gamma(|m|+1+s)}
=\frac{C_h(s)}{s} \frac{\Gamma(4\ell-s)}{\Gamma(4\ell+s)} =c_h(s) \frac{\Gamma(4\ell-s)}{\Gamma(4\ell+s)}
$$
and hence by \eqref{q22} and Proposition \ref{prop1} (3)
$$
\lim_{s\to0^+}\frac{D^h_\ell(s)-1}{s}=-2(\psi(4\ell)+\gamma)-\log h^2=-\sum_{0<|m-n|<4\ell}\frac{1}{|n-m|}-\log h^2=d_\ell^h,
$$
where we used the property of $\psi$ as follows
$$\psi(4\ell)=\sum_{m=1}^{4\ell-1}\frac1m-\gamma=\frac12\sum_{0<|m-n|<4\ell}\frac{1}{|n-m|}-\gamma.$$

\textbf{Step 6}: add $d_\ell^h f_h(n)$ to $a^h_\ell(n)$.
Taking $a^h_\ell(n)$ and $d_\ell^h f_h(n)$ into consideration, we arrive at
\begin{align*}
a^h_\ell(n)+d_\ell^h f_h(n)
&=\sum_{0<|m-n|<4\ell}\frac{f_h(n)-f_h(m)}{|n-m|}-\sum_{0<|m-n|<4\ell}\frac{f_h(n)}{|n-m|}-(\log h^2)f_h(n)\\
%&=-\sum_{0<|m-n|<4\ell}\frac{f_h(m)}{|n-m|}-(\log h^2)f_h(n)\\
&=-\sum_{m\in\zz, m\neq n}\frac{f_h(m)}{|n-m|} -(\log h^2)f_h(n)+\sum_{|m-n|\ge4\ell}\frac{f_h(m)}{|n-m|}  =\log(-\Delta_h)f_h(n) +g(n).
\end{align*}

\textbf{Step 7}: estimate $g$. Obviously, one has
$$
|g(n)|\le \|f_h\|_{\ell^1(\zz)}\ell^{-1}.
$$

\textbf{Step 8}: convergence in $L^\infty$-sense.\footnote{This $L^\infty$-convergence implies the pointwise convergence. }
By Steps 1 and 6, we decompose the desired result as
\begin{align*}
\left\|\frac{(-\Delta_h)^sf_h-f_h}{s}-\log(-\Delta_h)f_h\right\|_{\ell^\infty(\zz)}
%&=\left\|\frac{A^h_\ell(s,\cdot)+D^h_\ell(s)f_h-f_h}{s}-\log(-\Delta_h)f_h\right\|_{\ell^\infty(\zz)} \\
%&=\left\|\frac{A^h_\ell(s,\cdot)}{s}+\frac{D^h_\ell(s)f_h-f_h}{s}-\log(-\Delta_h)f_h\right\|_{\ell^\infty(\zz)} \\
%&=\left\|\frac{A^h_\ell(s,\cdot)}{s}-a^h_\ell+a^h_\ell+\frac{D^h_\ell(s)f_h-f_h}{s}-\log(-\Delta_h)f_h\right\|_{\ell^\infty(\zz)}\\
%&=\left\|\frac{A^h_\ell(s,\cdot)}{s}-a^h_\ell+a^h_\ell+\frac{D^h_\ell(s)f_h-f_h}{s}-d_\ell^hf_h+d_\ell^hf_h-\log(-\Delta_h)f_h\right\|_{\ell^\infty(\zz)}\\
&=\left\|\frac{A^h_\ell(s,\cdot)}{s}-a^h_\ell+\frac{D^h_\ell(s)f_h-f_h}{s}-d_\ell^hf_h+g\right\|_{\ell^\infty(\zz)} \\
%&\le \left\|\frac{A^h_\ell(s,\cdot)}{s}-a^h_\ell\right\|_{\ell^\infty(\zz)}+\left|\frac{D^h_\ell(s)-1}{s}-d_\ell^h\right| \|f_h\|_{\ell^\infty(\zz)}
% +\|g\|_{\ell^\infty(\zz)} \\
&\le \left\|\frac{A^h_\ell(s,\cdot)}{s}-a^h_\ell\right\|_{\ell^\infty(\zz\bigcap(-2\ell,2\ell))}+\left\|\frac{A^h_\ell(s,\cdot)}{s}-a^h_\ell\right\|_{\ell^\infty(\zz\setminus(-2\ell,2\ell))} \\
&\ +\left|\frac{D^h_\ell(s)-1}{s}-d_\ell^h\right| \|f_h\|_{\ell^\infty(\zz)} +\|g\|_{\ell^\infty(\zz)}.
\end{align*}
Therefore, invoke all estimates in Steps 2-5 and 7 to deduce
$$
\lim_{s\to0^+}\left\|\frac{(-\Delta_h)^sf_h-f_h}{s}-\log(-\Delta_h)f_h\right\|_{\ell^\infty(\zz)}
\le 3\|f_h\|_{\ell^1(\zz)}\ell^{-1}.
$$
As the integer $\ell\ge | \supp f|/h$ is arbitrary, we derive that the desired conclusion by letting $\ell\to+\infty$.

\textbf{Step 9}: $\log(-\Delta_h)f_h$ is in $\ell^\infty(\mathbb{Z})$.
Note that
$$\sum_{m\in\zz, m\neq n}\frac{|f_h(m)|}{|n-m|}\le \sum_{m\in\zz, m\neq n}|f_h(m)| = \|f_h\|_{\ell^1(\mathbb{Z})}$$
and hence
$$\|\log(-\Delta_h)f_h\|_{\ell^\infty(\mathbb{Z})}\le  \|f_h\|_{\ell^1(\mathbb{Z})}+|\log h^2|\,\|f_h\|_{\ell^\infty(\mathbb{Z})}<\infty.$$

\textbf{Step 10}: completion of the proof.
Finally, by Steps 8 and 9, we know that the pointwise formula \eqref{q2} holds,
$\log(-\Delta_h)f_h$ is in $\ell^\infty(\mathbb{Z})$,
and $[(-\Delta_h)^sf_h-f_h]/s$ approaches $\log(-\Delta_h)f_h$ in $L^\infty$-sense and also in pointwise sense as $s$ goes to $0^+$.
This completes the proof.
\end{proof}

\section{\bf The derivative of the fractional Laplacian on $\zz^N$}\label{s3}

\subsection{Preliminaries}
Our target in this subsection is to introduce the multidimensional discrete Laplacian
$$
\Delta_Nf(\n)=\sum_{k=1}^N\Delta_{N,k}f(\n)=\sum_{k=1}^N [f(\n+\mathbf{e}_k)+f(\n-\mathbf{e}_k)-2f(\n)],
$$
where $\{\mathbf{e}_k\}_{k=1}^N$ is an orthonormal basis in $\zz^N$ ($N\ge2$).
The heat semigroup $W_t=e^{t\Delta_N}$ is the solution mapping of the $N$-dimensional semidiscrete heat equation
$$
\begin{cases}
\partial_t u=\Delta_Nu, \ &\textrm{in }\zz^N\times(0,\infty),\\
u(\cdot,0)=f, &\textrm{on }\zz^N.
\end{cases}
$$
The solution to the above equation is given by (see Section \ref{s2} for the one-dimensional case)
$$
u(\n,t)=W_tf(\n)=\sum_{\m\in \zz^N}G_{t,N}(\n-\m)f(\m),
$$
where the heat kernel $G_{t,N}(\m)$ is constructed by
$$
G_{t,N}(\m)=\prod_{k=1}^N(e^{-2t}I_{m_k}(2t)),\quad \m=(m_1,\cdots,m_N),
$$
through the modified Bessel function of first kind $I_a$.
By (14) and (30) in \cite{Ci2023},
we know that the heat kernel $G_{t,N}(\m)$ admits the polynomial growth as
\begin{align}\label{p3}
0\le G_{t,N}(\m)\le C\left(\frac{\sqrt{t}}{\sqrt{t}+|\m|}\right)^2\frac{1}{(\sqrt{t}+|\m|)^N}.
\end{align}

%For any $0\le s\le1$, the weighted Lebesgue space, denoted by $\ell_{\pm s}(\zz)$, is defined as the class of all functions $f:\zz_h\to\mathbb{R}$ such that their norms
%\begin{align*}
%\|f\|_{\ell_{\pm s}(\zz)}=\sum_{n\in\zz}\frac{f_h(n)}{(1+|n|)^{1\pm 2s}}<\infty.
%\end{align*}
If $f\in \ell_{s}(\zz^N)$ with $0<s<1$, %(see \cite[(59)]{Ci2023} for the $N$-dimensional case and Section \ref{s2} for the one-dimensional case),
then
$$
(-\Delta_N)^sf(\mathbf{n})=\sum_{\m\in\zz^N,\m\neq \n}\mathcal{K}_s(\mathbf{n}-\mathbf{m})(f(\mathbf{n})-f(\mathbf{m})),
$$
where the discrete kernel $\mathcal{K}_s(\m)$ is given by
\begin{align}\label{q6}
\mathcal{K}_s(\m)=\frac{1}{|\Gamma(-s)|}\int_0^\infty G_{t,N}(\m)\frac{\d t}{t^{s+1}}\mathbbm{1}_{\zz^N\setminus\{\mathbf{0}\}} (\m) ;
\end{align}
see \cite[Section 5]{Ci2023} for more details.
Denote by ${K}(\m)$ the value of the quotient $\mathcal{K}_s(\m)/s$ at zero
\begin{align}\label{q3}
{K}(\m)
=\left.\frac{\mathcal{K}_s(\mathbf{m})}{s}\right|_{s=0}
=\left.\frac{1}{\Gamma(1-s)}\int_0^\infty G_{t,N}(\m)\frac{\d t}{t^{s+1}}\mathbbm{1}_{\zz^n\setminus\{\mathbf{0}\}} (\m)\right|_{s=0}
=\int_0^\infty G_{t,N}(\m)\frac{\d t}{t}\mathbbm{1}_{\zz^n\setminus\{\mathbf{0}\}} (\m)
\end{align}
for simplicity.
We refer the readers to \cite{CGRTV2017,MSW2002} and references therein for more discrete works
such as the Hardy-Littlewood function \cite{BMSW2019,LW2017}, Hardy space \cite{BD2014}, Calder\'{o}n reproducing formula \cite{HLY2001}, Radon transform \cite{ISMW2007}.

\subsection{\bf Estimates for discrete kernels}

This subsection establishes some upper bounds for ${\mathcal{K}_s(\m)}$ and ${K}(\m)$.

%\begin{align*}
%\mathcal{K}_s(\m)=\frac{1}{|\Gamma(-s)|}\int_0^\infty G_{t,N}(\n)\frac{\d t}{t^{s+1}}\mathbbm{1}_{\zz^n\setminus\{\mathbf{0}\}} (\n)\\
%{K}(\m)=\int_0^\infty G_{t,N}(\n)\frac{\d t}{t}\mathbbm{1}_{\zz^n\setminus\{\mathbf{0}\}} (\n)
%\end{align*}

\begin{prop}\label{prop2}
%The following formulas hold.
\begin{enumerate}
  \item  There exists a constant $C_N>0$ such that
  $$\displaystyle 0\le\mathcal{K}_s(\m)\le \frac{C_N}{|\Gamma(-s)|}\left(\frac{1}{1-s}+\frac{2}{N+2s}\right)
   \frac{\mathbbm{1}_{\zz^n\setminus\{\mathbf{0}\}} (\m)}{|\m|^{N+2s}}. $$ %,\quad \m\meq\mathbf{0}.$$
  \item  $\displaystyle \mathcal{K}_s(\m)=K(\m)s+o(s),\quad s\to0^+. $
  \item  There exists a constant $C_N>0$ same as in (1) such that
  $$0\le \displaystyle {K}(\m)\le C_N\left(1+\frac{2}{N}\right)\frac{\mathbbm{1}_{\zz^n\setminus\{\mathbf{0}\}} (\m)}{|\m|^{N}}. $$
\end{enumerate}
\end{prop}

\begin{proof}
(1) See \cite[Theorem 30]{Ci2023} for the proof.

(2) One can write
\begin{align*}
\left|\frac{\mathcal{K}_s(\m)}{s}-{K}(\m)\right|
&\le \left|\frac{1}{\Gamma(1-s)}\int_0^\infty G_{t,N}(\m)\frac{\d t}{t^{s+1}}-\int_0^\infty G_{t,N}(\m)\frac{\d t}{t}\right|\\
&\le \frac{1}{\Gamma(1-s)}\int_0^\infty G_{t,N}(\m)|{t^{-s}}-1|\frac{\d t}{t}+\left|\frac{1}{\Gamma(1-s)}-1\right| \int_0^\infty G_{t,N}(\m)\frac{\d t}{t}.
%&\le \frac{1}{\Gamma(1-s)}\int_0^\infty \frac{C}{t^{N/2}}\min\left\{1,\left(\frac{\sqrt{t}}{|\m|}\right)^{N+2}\right\}\left|\frac{1}{t^{s}}-1\right|\frac{\d t}{t}
% +\left|\frac{1}{\Gamma(1-s)}-1\right| \int_0^\infty G_{t,N}(\m)\frac{\d t}{t}
\end{align*}
On the one hand, we obtain by \eqref{p3} that
\begin{align*}
\int_0^\infty G_{t,N}(\m)|{t^{-s}}-1|\frac{\d t}{t}
%&\le C \int_0^1 \left(\frac{\sqrt{t}}{\sqrt{t}+|\m|}\right)^2\frac{1}{(\sqrt{t}+|\m|)^N}\left(\frac{1}{t^s}-1\right)\frac{\d t}{t}
% +C \int_1^\infty \left(\frac{\sqrt{t}}{\sqrt{t}+|\m|}\right)^2\frac{1}{(\sqrt{t}+|\m|)^N}\left(1-\frac{1}{t^s}\right)\frac{\d t}{t}  \\
&\le C \int_0^1 \frac{t}{|\m|^{N+2}}\left(\frac{1}{t^s}-1\right)\frac{\d t}{t}
 +C \int_1^\infty \frac{1}{t^{N/2}}\left(1-\frac{1}{t^s}\right)\frac{\d t}{t}  \\
&\le C \frac{s}{1-s} \frac{1}{|\m|^{N+2}}+C\frac{s}{1+s}\to0, \quad s\to0^+.
\end{align*}
On the other hand, it follows that
\begin{align*}
|\Gamma(1)-\Gamma(1-s)|
&\le \int_0^\infty e^{-t}|t-t^{1-s}|\frac{\d t}{t}\le \int_0^1 \left(\frac{1}{t^s}-1\right)\d t+2\int_1^\infty \left(\frac{1}{t^2}-\frac{1}{t^{2+s}}\right)\d t \\
&=\frac{s}{1-s}+2\frac{s}{1+s}\to 0, \quad s\to 0^+,
\end{align*}
and from \eqref{p3} that
\begin{align}\label{p4}
 \int_0^\infty G_{t,N}(\m)\frac{\d t}{t}\le C\int_0^\infty \left(\frac{\sqrt{t}}{\sqrt{t}+|\m|}\right)^2\frac{1}{(\sqrt{t}+|\m|)^N}\frac{\d t}{t}=\frac{C}{|\m|^N}.
\end{align}
Based on the above argument, we derive that
$$ \mathcal{K}_s(\m)=K(\m)s+o(s),\quad s\to0^+.$$

%
%Note the integrability of $G_{t,N}(\m)/t$ (see the proof of \cite[(77)]{Ci2023} for example), and $|{t^{-s}}-1|\le 2$ for all $t\ge1$.
%Then invoke the Lebesgue dominated convergence theorem to obtain the desired result.

(3)  The required result follows from \eqref{p4}. This completes the proof.
%\begin{align*}
%{K}(\m)\le \lim_{s\to0^+}\frac{\mathcal{K}_s(\m)}{s}
% \le  \lim_{s\to0^+}\frac{C_N}{|\Gamma(-s)|}\left(\frac{1}{1-s}+\frac{2}{N+2s}\right)\frac{\mathbbm{1}_{\zz^n\setminus\{\mathbf{0}\}} (\m)}{|\m|^{N+2s}}
% =C_N\left(1+\frac{2}{N}\right)\frac{\mathbbm{1}_{\zz^n\setminus\{\mathbf{0}\}} (\m)}{|\m|^{N}}.
%\end{align*}
\end{proof}

\subsection{\bf From one-dimension to higher-dimension}
In this subsection, we go further to consider the $N$-dimensional version of Theorem \ref{thm1} ($N\ge2$).

In order to generalize Theorem \ref{thm1} to the $N$-dimensional case,
the main difficulty arises from the lack of the explicit representation for the discrete kernel $\mathcal{K}_s(\m)$.
In Theorem \ref{thm1}, all constants derived from $\mathcal{K}^h_s(m)$ such as $C_h(s)$, $c_h^\prime(0)$, $D^h_{\ell}(s)$ and $d_\ell$ are clear and unambiguous.
However, in the $N$-dimensional setting, the discrete kernel $\mathcal{K}_s(\m)$ is determined by
\begin{align*}
\mathcal{K}_s(\mathbf{m})
&=\frac{1}{|\Gamma(-s)|}\int_0^\infty G_{t,N}(\m)\frac{\d t}{t^{s+1}}\mathbbm{1}_{\zz^n\setminus\{\mathbf{0}\}} (\m) \\
&=\frac{1}{|\Gamma(-s)|}\int_0^\infty \prod_{k=1}^N(e^{-2t}I_{m_k}(2t))\frac{\d t}{t^{s+1}}\mathbbm{1}_{\zz^n\setminus\{\mathbf{0}\}} (\m),\quad \m=(m_1,\cdots,m_N).
\end{align*}
We do not how to calculate accurately the product of $I_{m_k}(2t)$ with $k=1,\cdots,N$,
though it can be estimated as
$$
\frac{\Gamma({\frac{m_1+\cdots+m_N}{N}}+1)^N}{\Gamma(m_1+1)\cdots\Gamma(m_N+1)}\left(I_{\frac{m_1+\cdots+m_N}{N}}(2t)\right)^N
\le I_{m_1}(2t)\cdots I_{m_N}(2t)
\le \left(I_{\frac{m_1+\cdots+m_N}{N}}(2t)\right)^N
$$
for every $m_1,\cdots m_N>-1$; see \cite[Proposition 1]{Ci2023}.
With this estimate in hand, we can calculate the upper/lower bound for $G_{t,N}(\m)$ and hence for $\mathcal{K}_s(\mathbf{m})$,
but not the explicit representation.\footnote{We thank Prof. \'{O}. Ciaurri for reminding us this fact.}
Therefore we have to surmise some crucial constants similar as $c_h^\prime(0)$ and $d_\ell$ in Theorem \ref{thm1}.
In the following paragraph, we use the same notation as in the proof of Theorem \ref{thm1}.

The Step 1 in the proof of Theorem \ref{thm1} tells us that the term $D_\ell(s)$ in the $N$-dimensional setting should be defined by
$$
D_\ell(s)=\sum_{|\m|\ge4\ell}\mathcal{K}_s(\m).
$$
Since the discrete kernel $\mathcal{K}_s(\m)$ does not admit an explicit representation, we can not even verify that $D_\ell(0)$ is equal to one.
However, inspired by the Step 5 in the proof of Theorem \ref{thm1}, we expect that
$$
\frac{D_\ell(s)-1}{s}=-\sum_{0<|\m|<4\ell}K(\m)+o(1),\quad s\to 0^+.
$$
Unfortunately, this behaviour of $D_\ell(s)$ at zero is valid for $\zz$, not for $\zz^N$ ($N\ge2$).
Here, we should introduce a corrector to find the correct form of the above identity.
More precisely, hypothesize for the moment that
$$
\frac{D_\ell(s)-1}{s}=-\sum_{0<|\m|<4\ell}K(\m)+\rho_N+o(1)=d_\ell(s)+o(1),\quad s\to 0^+,
$$
 where the corrector $\rho_N$ depends on $N$ only to be determined later.
The main difficulty of showing the identity above arises from the lack of the explicit representation for the discrete kernel.
Note that
$$ \sum_{\m \in \mathbb{Z}^N} G_{t,N}(\m)=1$$
is valid by the conservation property of the heat kernel (see \cite[(21)]{Ci2023}).
A careful analysis gives rise to
\begin{align*}
\frac{D_\ell(s)-1}{s}
%&=\frac{1}{\Gamma(1-s)}\int_0^\infty \sum_{|\m|\ge4\ell} G_{t,N}(\m)\frac{\d t}{t^{s+1}}-\frac1s\\
&=\frac{1}{\Gamma(1-s)}\int_0^1 \sum_{|\m|\ge4\ell} G_{t,N}(\m)\frac{\d t}{t^{s+1}}
 +\frac{1}{\Gamma(1-s)}\int_1^\infty \sum_{|\m|\ge4\ell} G_{t,N}(\m)\frac{\d t}{t^{s+1}}-\frac1s  \\
&=\frac{1}{\Gamma(1-s)}\int_0^1 \sum_{|\m|\ge4\ell} G_{t,N}(\m)\frac{\d t}{t^{s+1}}
 -\frac{1}{\Gamma(1-s)}\int_1^\infty \sum_{|\m|<4\ell} G_{t,N}(\m)\frac{\d t}{t^{s+1}} +\frac{1}{s\Gamma(1-s)}-\frac1s  \\
&=\frac{1}{\Gamma(1-s)}\int_0^1 \sum_{|\m|\ge4\ell} G_{t,N}(\m)\frac{\d t}{t^{s+1}}
 -\frac{1}{\Gamma(1-s)}\int_1^\infty \sum_{|\m|<4\ell} G_{t,N}(\m)\frac{\d t}{t^{s+1}} +J_1  \\
&=\frac{1}{\Gamma(1-s)}\int_0^1 \sum_{|\m|\ge4\ell} G_{t,N}(\m)\frac{\d t}{t^{s+1}}-\int_0^1 \sum_{|\m|\ge4\ell} G_{t,N}(\m)\frac{\d t}{t}
 +\int_0^1 \sum_{|\m|\ge4\ell} G_{t,N}(\m)\frac{\d t}{t} \\
&\ -\frac{1}{\Gamma(1-s)}\int_1^\infty \sum_{|\m|<4\ell} G_{t,N}(\m)\frac{\d t}{t^{s+1}}+\int_1^\infty \sum_{|\m|<4\ell} G_{t,N}(\m)\frac{\d t}{t}
  -\int_1^\infty \sum_{|\m|<4\ell} G_{t,N}(\m)\frac{\d t}{t} +J_1 \\
&=J_2 +\int_0^1 \sum_{|\m|\ge4\ell} G_{t,N}(\m)\frac{\d t}{t} +J_3  -\int_1^\infty \sum_{|\m|<4\ell} G_{t,N}(\m)\frac{\d t}{t} +J_1 \\
%&=J_2 +\int_0^1 \sum_{\m\in\zz^N,\m\neq\mathbf{0}} G_{t,N}(\m)\frac{\d t}{t}-\int_0^1 \sum_{0<|\m|<4\ell} G_{t,N}(\m)\frac{\d t}{t} \\
%&\ +J_3  -\int_1^\infty \sum_{|\m|<4\ell} G_{t,N}(\m)\frac{\d t}{t} +J_1 \\
&=J_2 +\int_0^1 \sum_{\m\in\zz^N,\m\neq\mathbf{0}} G_{t,N}(\m)\frac{\d t}{t}-\int_0^1 \sum_{0<|\m|<4\ell} G_{t,N}(\m)\frac{\d t}{t} \\
&\ +J_3  -\int_1^\infty G_{t,N}(\mathbf{0})\frac{\d t}{t}-\int_1^\infty \sum_{0<|\m|<4\ell} G_{t,N}(\m)\frac{\d t}{t} +J_1 \\
&=- \sum_{0<|\m|<4\ell} K(\m)+\sum_{\m\in\zz^N,\m\neq\mathbf{0}}\int_0^1 G_{t,N}(\m)\frac{\d t}{t}-\int_1^\infty G_{t,N}(\mathbf{0})\frac{\d t}{t} +J_1+J_2+J_3.
\end{align*}
On the one hand, we claim that the last two terms $J_2$ and  $J_3$ converge to zero as $s\to0^+$.
Indeed, by \eqref{p3}, we have
$$
\sum_{|\m|\ge4\ell} G_{t,N}(\m)
\le C\sum_{|\m|\ge4\ell} \left(\frac{\sqrt{t}}{\sqrt{t}+|\m|}\right)^2\frac{1}{(\sqrt{t}+|\m|)^N}
\le C\sum_{|\m|\ge4\ell} \frac{t}{|\m|^{N+2}}\le C\frac{t}{\ell^2}\le Ct
$$
and hence
\begin{align*}
|J_2|
&\le\frac{1}{\Gamma(1-s)}\int_0^1 \sum_{|\m|\ge4\ell} G_{t,N}(\m)\left(\frac{1}{t^{s}}-1\right)\frac{\d t}{t}
 +\left|\frac{1}{\Gamma(1-s)}-1\right|\int_0^1 \sum_{|\m|\ge4\ell} G_{t,N}(\m)\frac{\d t}{t} \\
&\le \frac{C}{\Gamma(1-s)}\int_0^1 \left(\frac{1}{t^{s}}-1\right)\d t
 +C\left|\frac{1}{\Gamma(1-s)}-1\right|\int_0^1 \d t \\
& = \frac{C}{\Gamma(1-s)}\frac{s}{1-s}+C\left|\frac{1}{\Gamma(1-s)}-1\right|\to0, \quad s\to0^+.
\end{align*}
The proof of the term $J_3$ is similar, and is left to the readers.
On the other hand, the first term $J_1$ can be estimated as
$$
\lim_{s\to0^+} J_1=\lim_{s\to0^+}\frac{1-\Gamma(1-s)}{s\Gamma(1-s)}=\lim_{s\to0^+}\frac{1-\Gamma(1-s)}{s}=\lim_{s\to0^+}\Gamma^\prime(1-s)=\Gamma^\prime(1)=-\gamma
$$
by the L'Hospital's rule.
Therefore the corrector $\rho_N$ can be chosen as
\begin{align}\label{q4}
\rho_N=\sum_{\m\in\zz^N,\m\neq\mathbf{0}}\int_0^1 G_{t,N}(\m)\frac{\d t}{t}-\int_1^\infty G_{t,N}(\mathbf{0})\frac{\d t}{t}-\gamma.
\end{align}
The remainder of the argument is easy and analogous to that in Theorem \ref{thm1}, and is left to the readers.

Based on the above argument, we derive the second result of this note.
\begin{thm}\label{thm2}
Let $f$ be a real-value function on $\zz^N$ with bounded support.
Then we have
$$
\log(-\Delta_N)f(\n)=\left.\frac{\d }{\d s}(-\Delta_N)^s\right|_{s=0^+}f(\n)=-\sum_{\m\in\zz^N, \m\neq \n} K(\n-\m)f(\m)+\rho_Nf(\n),
$$
where the discrete kernel $K(\m)$ and the corrector $\rho_N$ are as in  \eqref{q3} and \eqref{q4}, respectively.
Moreover, we have $\log(-\Delta_N)f\in \ell^\infty(\mathbb{Z}^N)$ and
$$\lim_{s\to0^+}\frac{(-\Delta_N)^sf-f}{s}=\log(-\Delta_N)f$$
in $L^\infty$-sense.
\end{thm}

\begin{rem}
In the Euclidean space $\mathbb{R}^d$, Chen-Weth \cite{CW2019} proved that
\begin{align}\label{q5}
\left.\frac{\d}{\d s}(-\Delta)^s\right|_{s=0^+} f(x)=c_d\int_{\mathbb{R}^d}\frac{f(x)\mathbbm{1}_{B_1(x)}(y)-f(y)}{|x-y|^{d}}\d y+\rho_df(x);
\end{align}
see Theorem A.
The observant reader might notice this result in the continuous setting is more or less different from our Theorem \ref{thm2}.
The term $f(x)\mathbbm{1}_{B_1(x)}(y)$ appears in \eqref{q5},
but there is no the corresponding discrete term in our result,
and hence the derivative of the fractional discrete Laplacian can be regard as an exotic discrete Riesz potential.
In fact, on the one hand, the open ball $B_1(\n)=\{\m\in\zz^N:\ |\m-\n|<1\}$ in the discrete setting tells us $\m=\n$.
On the other hand, the subscript in the summation indicates $\m\neq\n$.
Therefore we reformulate $\log(-\Delta_N)f(\n)$ in some sense as
$$
\log(-\Delta_N)f(\n)=\sum_{\m\in\zz^N, \m\neq \n} K(\n-\m)(f(\n)\mathbbm{1}_{B_1(\n)}(\m)-f(\m))+\rho_Nf(\n),
$$
which is the discrete version of \eqref{q5}.
\end{rem}

\section{\bf A final remark: the discrete Riesz potential}\label{s4}

In Sections \ref{s2} and \ref{s3}, we consider the asymptotic behavior for the \textit{positive} power of the discrete Laplacian
\begin{align}\label{q7}
(-\Delta_N)^sf(\n)=f(\n)+[-K\ast f(\n)+\rho_Nf(\n)]s+o(s),\quad s\to0^+.
\end{align}
In fact, we can take the power $s$ to be negative.
The \textit{negative} power of the discrete Laplacian is well known as the discrete Riesz potential.
It can be defined by the heat semigroup
$$
(-\Delta_N)^s=\frac{1}{\Gamma(-s)}\int_0^\infty e^{t\Delta_N}\frac{\d t}{t^{s+1}},\quad -N/2<s<0.
$$
When $N=1$, the discrete  kernel of $(-\Delta_h)^s$ with $h>0$ can be expressed as
$$
\mathcal{K}_s(m)=\pi^{-1/2}(-s)\left(\frac{2}{h}\right)^{2s}\frac{\Gamma(1/2+s)}{\Gamma(1-s)}\frac{\Gamma(|m|-s)}{\Gamma(|m|+1+s)} \mathbbm{1}_{\zz\setminus\{0\}}(m)
$$
through the Gamma function; see \cite[Theorem 1.3]{CRSTV2018} for more details.
When $N\ge2$, the discrete kernel of $(-\Delta_N)^s$ is not explicit, but given by
$$
\mathcal{K}_s(\m)=\frac{1}{\Gamma(-s)}\int_0^\infty G_{t,N}(\m)\frac{\d t}{t^{s+1}}\mathbbm{1}_{\zz^N\setminus\{\mathbf{0}\}} (\m)
$$
similar to \eqref{q6}; see \cite[Theorem 14]{Ci2023} for more details.

A basic question arises from the discrete Riesz potential above:
\begin{enumerate}
  \item[$\bullet$] Question:  Can we derive the asymptotic behavior for the \textit{negative} power of the discrete Laplacian similar as \eqref{q7}?
  %\item[$\bullet$] Question 2: Can we introduce a new Morrey function via the variable Gaussian heat kernel on a general underlying space?
%  \item[$\bullet$] Question 3: Can we introduce a new Morrey function via the variable Gaussian heat kernel on a general underlying space?
\end{enumerate}

Fortunately, the answer for this question is positive and is the same as the positive power case; see \cite{Ch2023} for the continuous case.
Here, we present the corresponding result for $(-\Delta_N)^s$ with $-N/2<s<0$ only ,
whose proof is analogous to that in Theorems \ref{thm1} and \ref{thm2}, and left to the readers.

\begin{thm}\label{thm3}
Let $f$ be a real-value function on $\zz^N$ ($N\ge1$) with bounded support.
Then we have
$$
\log(-\Delta_N)f(\n)=\left.\frac{\d }{\d s}(-\Delta_N)^s\right|_{s=0^-}f(\n)=-\sum_{\m\in\zz^N, \m\neq \n} K(\n-\m)f(\m)+\rho_Nf(\n),
$$
where the discrete kernel $K(\m)$ and the corrector $\rho_N$ are as in  \eqref{q3} and \eqref{q4}, respectively.
Moreover, we have $\log(-\Delta_N)f\in \ell^\infty(\mathbb{Z}^N)$ and
$$\lim_{s\to0^-}\frac{(-\Delta_N)^sf-f}{s}=\log(-\Delta_N)f$$
in $L^\infty$-sense.
\end{thm}

{\small{
\noindent {\textbf{Acknowledgements.}}
The authors are grateful for the anonymous referee's comments on the original version of this paper.
Bo Li thanks Prof. \'{O}. Ciaurri and  L. Roncal for many helpful discussions about Gamma functions during the preparation of this work.
Bo Li was supported by NNSF of China (12471094,12201250), Zhejiang NSF of China (LQ23A010007), Jiaxing NSF of China (2023AY40003) and Qinshen Scholar Program of Jiaxing University.
Qingze Lin is supported by Guangdong Basic and Applied Basic Research Foundation (2024A1515110227) and STU Scientific Research Initiation Grant (NTF24015T).
Huoxiong Wu was supported by NNSF of China (12171399 \& 12271041).
}}

\noindent
{\bf\large Declarations}\\

\noindent
{\bf Author Contributions}\ \  B. Li, Q. Lin and H. Wu participated equally in the work. All authors reviewed the manuscript.\\

\noindent
{\bf Competing Interests}\ \  The authors declare no competing interests.\\

\noindent
{\bf Ethical Approval}\ \  Not applicable.\\

\noindent
{\bf Availability of data and material}\ \  Data sharing not applicable to this article as no datasets were generated or analyzed during the current study.

 %%%%%%%%%%%%%%%%%%%%%%%%%%%%%%%%%%%%%%%%%%%%%%%%%%%%%%

%%%%%%%%%%%%%%%%%%%%%%%%%%%%%%%%%%%%%%%%%%%%%%%%%%%%%%


\begin{thebibliography}{10}


\bibitem{ACDFGW2019}
W. Ao, H. Chan, A. DelaTorre, M. Fontelos, M. Gonz\'{a}lez, J. Wei,
On higher-dimensional singularities for the fractional Yamabe problem: a nonlocal Mazzeo-Pacard program.
Duke Math. J. 168 (2019), no. 17, 3297-3411.
%
%\bibitem{ACL2022}
%A. Arenas, \'{O}. Ciaurri, E. Labarga,
%Discrete harmonic analysis associated with Jacobi expansions II: the Riesz transform.
%Potential Anal. 57 (2022), no. 4, 501-520.

\bibitem{BFS2018}
B. Barrios, A. Figalli, X. Ros-Oton,
Free boundary regularity in the parabolic fractional obstacle problem. Comm. Pure Appl. Math. 71 (2018), no. 10, 2129-2159.


\bibitem{BMSW2019}
J. Bourgain, M. Mirek, E. Stein, B. Wr\'{o}bel,
Dimension-free estimates for discrete Hardy-Littlewood averaging operators over the cubes in $\zz^d$.
Amer. J. Math. 141 (2019), no. 4, 857-905.

%\bibitem{Bu2023}
%T. Bui,
%Maximal regularity of parabolic equations associated with a discrete Laplacian.
%J. Differential Equations 375 (2023), 277-303.
%
%
%
%\bibitem{Bu2014}
%T. Bui,
%Weighted Hardy spaces associated to discrete Laplacians on graphs and applications.
%Potential Anal. 41 (2014), no. 3, 817-848.


\bibitem{BD2014}
T. Bui, X. Duong,
Hardy spaces associated to the discrete Laplacians on graphs and boundedness of singular integrals.
Trans. Amer. Math. Soc. 366 (2014), no. 7, 3451-3485.


\bibitem{CSS2008}
L. Caffarelli, S. Salsa, L. Silvestre,
Regularity estimates for the solution and the free boundary of the obstacle problem for the fractional Laplacian.
Invent. Math. 171 (2008), no. 2, 425-461.


\bibitem{CS2007}
L. Caffarelli, L. Silvestre,
An extension problem related to the fractional Laplacian.
Comm. Partial Differential Equations 32 (2007), no. 7-9, 1245-1260.


\bibitem{CG2011}
S. Chang, M. Gonz\'{a}lez,
Fractional Laplacian in conformal geometry.
Adv. Math. 226 (2011), no. 2, 1410-1432.


\bibitem{Ch2023}
H. Chen,
Taylor expansions of Riesz convolution and the fractional Laplacians with respect to the order.
arXiv:2307.06198.


\bibitem{CW2019}
H. Chen, T. Weth,
The Dirichlet problem for the logarithmic Laplacian.
Comm. Partial Differential Equations 44 (2019), no. 11, 1100-1139.


\bibitem{CW2021}
H. Chen, T. Weth,
The Poisson problem for the fractional Hardy operator: distributional identities and singular solutions.
Trans. Amer. Math. Soc. 374 (2021), no. 10, 6881-6925.


%\bibitem{CL2022-2}
%T. Chen, F. Liu,
%Endpoint Sobolev regularity of the fractional maximal commutators.
%J. Fourier Anal. Appl. 28 (2022), no. 5, Paper No. 74, 39 pp.
%
%\bibitem{CL2022-1}
%T. Chen, F. Liu,
%Derivative bounds and continuity of maximal commutators.
%Studia Math. 266 (2022), no. 1, 93-119.

\bibitem{CFGW2021}
W. Chen, Z. Fu, L. Grafakos, Y. Wu,
Fractional Fourier transforms on $L^p$ and applications.
Appl. Comput. Harmon. Anal. 55 (2021), 71-96.


\bibitem{CM2023}
W. Chen, L. Ma,
Qualitative properties of solutions for dual fractional nonlinear parabolic equations. J. Funct. Anal. 285 (2023), no. 10, Paper No. 110117, 32 pp.

\bibitem{CDH2016}
Y. Chen, Y. Ding, G. Hong,
Commutators with fractional differentiation and new characterizations of BMO-Sobolev spaces. Anal. PDE 9 (2016), no. 6, 1497-1522.




\bibitem{Ci2023}
\'{O}. Ciaurri,
Harmonic analysis for a multidimensional discrete Laplacian.
arXiv:2312.16642.

\bibitem{CGRTV2017}
\'{O}. Ciaurri, T. Gillespie, L. Roncal, J. Torrea, J. Varona,
Harmonic analysis associated with a discrete Laplacian. J. Anal. Math. 132 (2017), 109-131.


\bibitem{CRSTV2018}
\'{O}. Ciaurri, L. Roncal, P. Stinga, J. Torrea, J. Varona,
Nonlocal discrete diffusion equations and the fractional discrete Laplacian, regularity and applications.
Adv. Math. 330 (2018), 688-738.

%
%
%\bibitem{DG2018}
%A. DelaTorre, M. Gonz\'{a}lez,
%Isolated singularities for a semilinear equation for the fractional Laplacian arising in conformal geometry.
%Rev. Mat. Iberoam. 34 (2018), no. 4, 1645-1678.
%
%\bibitem{FGLWY2023}
%Z. Fu, L. Grafakos, Y. Lin, Y. Wu, S. Yang,
%Riesz transform associated with the fractional Fourier transform and applications in image edge detection.
%Appl. Comput. Harmon. Anal. 66 (2023), 211–235.
%
%
%\bibitem{DLYYZ2023}
%F. Dai, X. Lin, C. Yang, W. Yuan, Y. Zhang,
%Brezis-Van Schaftingen-Yung formulae in ball Banach function spaces with applications to fractional Sobolev and Gagliardo-Nirenberg inequalities.
%Calc. Var. Partial Differential Equations 62 (2023), no. 2, Paper No. 56, 73 pp.

\bibitem{DFX2018}
B. Dong, Z. Fu, J. Xu,
Riesz-Kolmogorov theorem in variable exponent Lebesgue spaces and its applications to Riemann-Liouville fractional differential equations.
Sci. China Math. 61 (2018), no. 10, 1807-1824.

\bibitem{HLY2001}
Y. Han, S. Lu, D. Yang,
Inhomogeneous discrete Calder\'{o}n reproducing formulas for spaces of homogeneous type.
J. Fourier Anal. Appl. 7 (2001), no. 6, 571-600.


\bibitem{HT2024}
X. Hu, L. Tang,
Higher regularity of the free boundary in the obstacle problem for the fractional heat operator.
J. Funct. Anal. 286 (2024), no. 4, Paper No. 110274.


\bibitem{ISMW2007}
A. Ionescu, E. Stein, A. Magyar, S. Wainger,
Discrete Radon transforms and applications to ergodic theory.
Acta Math. 198 (2007), no. 2, 231-298.


\bibitem{JLX2014}
T. Jin, Y. Li, J. Xiong,
On a fractional Nirenberg problem, part I: blow up analysis and compactness of solutions.
J. Eur. Math. Soc. (JEMS) 16 (2014), no. 6, 1111-1171.

\bibitem{JX2014}
T. Jin, J. Xiong,
A fractional Yamabe flow and some applications.
J. Reine Angew. Math. 696 (2014), 187-223.

%\bibitem{JX2015}
%T. Jin, J. Xiong,
%On a fractional Nirenberg problem, Part II: Existence of solutions.
%Int. Math. Res. Not. IMRN 2015, no. 6, 1555-1589.
%
%\bibitem{KL1998}
%J. Kinnunen, P. Lindqvist,
%The derivative of the maximal function.
%J. Reine Angew. Math. 503 (1998), 161-167.

\bibitem{LHZ2023}
P. Li, R. Hu, Z. Zhai,
Strengthened fractional Sobolev type inequalities in Besov spaces.
Potential Anal. 59 (2023), no. 4, 2105-2121.

%\bibitem{LHZ2022}
%P. Li, R. Hu, Z. Zhai,
%Fractional Besov trace/extension-type inequalities via the Caffarelli-Silvestre extension.
%J. Geom. Anal. 32 (2022), no. 9, Paper No. 236, 30 pp.

\bibitem{LLLMS}
B. Li, J. Li, Q. Lin, B. Ma, T. Shen, A revisit to ``On BMO and Carleson measures on Riemannian manifolds''.
Proc. Roy. Soc. Edinburgh Sect. A, 154  (2024), no. 4, 1281–1307.
%https://doi.org/10.1017/prm.2023.58
%
%154, 1281–1307

%\bibitem{LW2019}
%F. Liu, H. Wu,
%A note on the endpoint regularity of the discrete maximal operator.
%Proc. Amer. Math. Soc. 147 (2019), no. 2, 583-596.


\bibitem{LW2017}
F. Liu, H. Wu,
Regularity of discrete multisublinear fractional maximal functions. Sci. China Math. 60 (2017), no. 8, 1461-1476.


\bibitem{LSX2024}
L. Liu, Y. Sun, J. Xiao,
Quasilinear Laplace equations and inequalities with fractional orders.
Math. Ann. 388 (2024), no. 1, 1-60.

%\bibitem{LWYY2019}
%L. Liu, S. Wu, D. Yang, W. Yuan,
%New characterizations of Morrey spaces and their preduals with applications to fractional Laplace equations.
%J. Differential Equations 266 (2019), no. 8, 5118-5167.
%
%
%
%\bibitem{LX2022}
%L. Liu, J. Xiao,
%Divergence \& curl with fractional order.
%J. Math. Pures Appl. (9) 165 (2022), 190-231.
%
%
%
%\bibitem{LXY2020}
%F. Liu, Q. Xue, K. Yabuta,
%Regularity and continuity of the multilinear strong maximal operators.
%J. Math. Pures Appl. (9) 138 (2020), 204-241.


\bibitem{MSW2002}
A. Magyar, E. Stein, S. Wainger,
Discrete analogues in harmonic analysis: spherical averages.
Ann. of Math. (2) 155 (2002), no. 1, 189–208.

\bibitem{SW2000}
E. Stein, S. Wainger,
Discrete analogues in harmonic analysis. II. Fractional integration.
J. Anal. Math. 80 (2000), 335-355.

\bibitem{ST2010}
P. Stinga, J. Torrea,
Extension problem and Harnack's inequality for some fractional operators,
Comm. Partial Differential Equations 35 (2010), no. 11, 2092-2122.

%\bibitem{Zh2015}
%Y. Zhou,
%Fractional Sobolev extension and imbedding.
%Trans. Amer. Math. Soc. 367 (2015), no. 2, 959-979.

%
%\bibitem{BCFR2010}
%J. Betancor, A. Chicco Ruiz, J. Fari\~{n}a, L. Rodríguez-Mesa,
%Odd BMO$(\mathbb{R})$ functions and Carleson measures in the Bessel setting.
%{\it Integral Equations Operator Theory} {\bf 66} (2010), no. 4, 463-494.
%
%\bibitem{BB17}
%S. Biagi, A. Bonfiglioli,
%The existence of a global fundamental solution for homogeneous H\"ormander operators via a global lifting method.
%{\it Proc. Lond. Math. Soc. (3)} {\bf 114} (2017), no. 5, 855-889.
%
%\bibitem{BLU}
%A. Bonfiglioli, E. Lanconelli, F. Uguzzoni,
%\textit{Stratified Lie groups and potential theory for their sub-Laplacians}. Springer Monographs in Mathematics. Springer, Berlin, 2007. xxvi+800 pp.
%
%\bibitem{CY}
%H. Cao, S. Yau, Gradient estimates, Harnack inequalities and estimates for heat kernels of the sum of squares of vector fields.
%{\it Math. Z.} {\bf 211} (1992), no. 3, 485-504.
%
%\bibitem{CJKS}
%T. Coulhon, R. Jiang, P. Koskela, A. Sikora,
%Gradient estimates for heat kernels and harmonic functions.
%{\it J. Funct. Anal.} {\bf 278} (2020), no. 8, 108398, 67 pp.
%
%
%\bibitem{DGMTZ}
%J. Dziuba\'{n}ski, G. Garrig\'{o}s, T. Martínez, J. Torrea, J. Zienkiewicz,
%BMO spaces related to Schr\"{o}dinger operator with potential satisfying reverse H\"{o}lder inequality, {\it Math. Z.} {\bf 249} (2005) 329-356.
%
%\bibitem{DYZ}
%X. Duong, L. Yan, C. Zhang,
%On characterization of Poisson integrals of Schrödinger operators with BMO traces.
%\textit{J. Funct. Anal.} \textbf{266} (2014), no. 4, 2053-2085.
%
%\bibitem{FJN} E. Fabes, R. Johnson and U. Neri,
%Spaces of harmonic functions representable by Poisson integrals of functions in BMO and $L_{p, \lambda}$.
%\textit{Indiana Univ. Math. J.} \textbf{ 25} (1976), 159--170.
%
%\bibitem{FN}
%E. Fabes, U. Neri,
%Characterization of temperatures with initial data in BMO.
%\textit{Duke Math. J.} \textbf{42} (1975), no. 4, 725-734.
%
%\bibitem{FN1}
%E. Fabes, U. Neri,
%Dirichlet problem in Lipschitz domains with BMO data.
%\textit{Proc. Amer. Math. Soc.} \textbf{78} (1980), 33--39.
%
%\bibitem{FS}
%C. Fefferman, E. Stein,
%$H^p$ spaces of several variables.
%{\it Acta Math}. {\bf 129} (1972), 137--195.
%
%\bibitem{Fol75}
%G. Folland, Subelliptic estimates and function spaces on nilpotent Lie groups.
%{\it Ark. Mat. }{\bf 13} (1975), no. 2, 161-207.
%
%\bibitem{Fol77}
%G. Folland, On the Rothschild-Stein lifting theorem.
%{\it Comm. Partial Differential Equations}{\bf 2} (1977), no. 2, 165-191.
%
%\bibitem{FS82}
%G. Folland, E. Stein,
%{\it Hardy spaces on homogeneous groups.}
%Mathematical Notes, 28. Princeton University Press, Princeton, N.J.; University of Tokyo Press, Tokyo, 1982. xii+285 pp.
%
%\bibitem{Hor67}
%L. H\"ormander, Hypoelliptic second order differential equations.
%{\it Acta Math.} {\bf 119} (1967), 147-171.
%
%\bibitem{Jerison}
%D. Jerison and A. S\'{a}nchez-Calle, Estimates for the
%heat kernel for a sum of squares of vector fields,
%Indiana Univ. Math. J. {\bf 35} (1986), 835-854.
%
%\bibitem{JL}
%R. Jiang, B. Li,
%Dirichlet problem for the Schr\"odinger equation with boundary value in BMO space.
%{\it Sci. China Math.} {\bf 65} (2022), no. 7, 1431-1468.
%
%\bibitem{LH}
%H. Li,
%Estimations $L^p$ des op\'{e}rateurs de Schr\"odinger sur les groupes nilpotents,
%\textit{J. Funct. Anal.} {\bf 161} (1999) 152-218.
%
%\bibitem{LY}
%P. Li, S. Yau,
%On the parabolic kernel of the Schrödinger operator.
%{\it Acta Math.} {\bf 156} (1986), no. 3-4, 153-201.
%
%\bibitem{LL}
%C. Lin, H. Liu,
%$BMO_L(\mathbb{H}^n)$ spaces and Carleson measures for Schr\"odinger operators.
%\textit{Adv. Math.} {\bf 228} (2011), no. 3, 1631-1688.
%
%\bibitem{LIN}
%Q. Lin, Gradient estimates for Schr\"odinger operators with characterizations of $BMO_{\mathcal{L}}$ on Heisenberg groups.
%\textit{ Proc. Amer. Math. Soc.} {\bf 151} (2023), no. 5, 2127-2142.
%
%\bibitem{RS76}
%L. Rothschild, E. Stein, Hypoelliptic differential operators and nilpotent groups.
%{\it Acta Math.} {\bf 137} (1976), no. 3-4, 247-320.
%
%\bibitem{St1970}
%E. Stein,
%\textit{Topics in Harmonic Analysis Related to the Littlewood-Paley Theory},
%{Annals of Mathematics Studies} \textbf{63}, Princeton Univ. Press, Princeton, NJ, 1970.
%
%
%
%
%\bibitem{St1993}
%E. Stein,
%{\it Harmonic analysis: Real variable methods, orthogonality and oscillatory integrals},
%Princeton Univ. Press, Princeton, NJ, (1993).
%
%
%\bibitem{SW2000}
%E. Stein, S. Wainger,
%Discrete analogues in harmonic analysis. II. Fractional integration,
%\textit{ J. Anal. Math.} {\bf 80} (2000), 335-355.

\end{thebibliography}
\end{document}